\def\BibLaTeXCommands{\nocite{*}}\def\loadTIKZ{\usepackage{tikz}\usetikzlibrary{matrix,arrows,shapes.arrows,chains,decorations.pathmorphing,decorations.pathreplacing,decorations.text,calc,trees,intersections,cd}}\newcommand\DocumentClassOptions{article,12pt,oneside} \newcommand\GeometryOPTIONS{hcentering=true,textheight=220mm,textwidth=150mm}\ifdefined\headpresent\else \providecommand\DocumentClassOptions{article}
\ifdefined\presentationON \documentclass[\DocumentClassOptions,20pt]{amsart}\usepackage{extsizes} \else \documentclass[\DocumentClassOptions]{amsart}\fi 
\usepackage{ifpdf}\ifdefined\BibLaTeX \usepackage[\BibLaTeX]{biblatex}\DeclareFieldFormat{labelalpha}{\thefield{entrykey}}\makeatletter\protected\def\abx@missing#1{\mbox{\reset@font\color{red}#1??}}\makeatother \DeclareFieldFormat{title}{\textit{#1}}\DeclareFieldFormat{year}{\textbf{#1}}\renewbibmacro*{date}{\printtext[year]{\printdate}}\BibLaTeXCommands \else\fi \newcommand\DIRroot{../../../../../../../../../../../}\overfullrule=5pt \newcommand\hfuzzReset{\hfuzz=3pt}\hfuzzReset \newcommand\toleranceReset{\tolerance=1400}\toleranceReset \newcommand\emergencystretchReset{\emergencystretch=1ex}\emergencystretchReset \hbadness=10000 \usepackage{xifthen}\usepackage{forarray}\usepackage{xstring}\usepackage{stringstrings}\def\StackCreate#1#2#3{\expandafter\def\csname#1\endcsname{#3}\expandafter\def\csname#1Push\endcsname##1{\expandafter\edef\csname#1\endcsname{##1#2\csname#1\endcsname}}\expandafter\def\csname TopAux#1\endcsname ##1#2##2#3{##1}\expandafter\def\csname#1Top\endcsname{\expandafter\expandafter\expandafter\expandafter\expandafter\expandafter\csname TopAux#1\endcsname\csname#1\endcsname}\expandafter\def\csname PopAux#1\endcsname ##1#2##2#3##3#2{\expandafter\def\csname##3\endcsname{##2#3}}\expandafter\def\csname#1Pop\endcsname{\expandafter\expandafter\expandafter\expandafter\expandafter\expandafter\csname PopAux#1\endcsname\csname#1\endcsname#1#2}}\def\GetAfterColonAux#1:#2;{#2}\def\GetAfterColon#1{\IfBeginWith{#1}{:}{\GetAfterColonAux#1;}{#1}}\usepackage{aliascnt}\newlength{\tempwidth}\newcommand{\fillX}[2][]{\settowidth{\tempwidth}{#2}\def\temp{#1}\ifx\temp\empty\else\addtolength{\tempwidth}{#1}\fi\leavevmode\cleaders\hbox to \tempwidth{\hss #2\hss }\hfill\kern0pt }\usepackage[shortlabels,inline]{enumitem}\setenumerate[1]{leftmargin=5.5ex}\setitemize[1]{leftmargin=5.5ex}\SetEnumitemKey{noindent}{leftmargin=0ex, itemindent=5ex, align=right, itemsep=1ex }\newcounter{ManualLabel}\makeatletter \newcommand\itemPatch[1][]{\item[\theenumi#1]\refstepcounter{ManualLabel}\def\@currentlabel{\theenumi#1}}\newcommand\itemDescribe[1][]{\item[#1]\refstepcounter{ManualLabel}\def\@currentlabel{#1}}\makeatother \ifdefined\presentationON \usepackage[hcentering=true,textheight=240mm,textwidth=210mm]{geometry} \else \ifdefined\GeometryOPTIONS \usepackage[\GeometryOPTIONS]{geometry}\fi \fi \ifdefined\selectPages \usepackage[\selectPages]{pagesel} \fi \usepackage{everypage-1x}\ifdefined\AddEverypageHook \newcommand\AddPrivateToMargin[1]{\AddEverypageHook{\tikz[overlay,remember picture]{\node at ($(current page.west)+(1.5,0)$) [rotate=90] {\textcolor{orange}{\vbox{\hrule width \the\textwidth height 0.5pt} \textcolor{defaultcolor}{#1}\ \vbox{\hrule width 40em height 0.5pt}}}; }}}\newcommand\AddLongversionToMargin[1]{\AddEverypageHook{\tikz[overlay,remember picture]{\node at ($(current page.west)+(2,0)$) [rotate=90] {\textcolor{\LongColor}{\vbox{\hrule width \the\textwidth height 0.5pt} \textcolor{defaultcolor}{#1}\ \vbox{\hrule width 40em height 0.5pt}}}; }}}\newcommand\AddOldversionToMargin[1]{\AddEverypageHook{\tikz[overlay,remember picture]{\node at ($(current page.west)+(2.5,0)$) [rotate=90] {\textcolor{\OldColor}{\vbox{\hrule width \the\textwidth height 0.5pt} \textcolor{defaultcolor}{#1}\ \vbox{\hrule width 40em height 0.5pt}}}; }}}\newcommand\AddLineToMargin[3]{\AddEverypageHook{\tikz[overlay,remember picture]{\node at ($(current page.west)+(#2,0)$) [rotate=90] {\textcolor{#1}{\vbox{\hrule width \the\textwidth height 0.5pt} \textcolor{defaultcolor}{#3}\ \vbox{\hrule width 40em height 0.5pt}}}; }}}\else\fi \providecommand\AddPrivateToMargin[1]{\AddToHook {shipout/background}{\tikz[overlay,remember picture]{\node at ($(current page.west)+(1.5,0)$) [rotate=90] {\textcolor{orange}{\vbox{\hrule width \the\textwidth height 0.5pt} \textcolor{black}{#1}\ \vbox{\hrule width 40em height 0.5pt}}}; }}}\providecommand\AddLongversionToMargin[1]{\AddToHook {shipout/background}{\tikz[overlay,remember picture]{\node at ($(current page.west)+(2,0)$) [rotate=90] {\textcolor{\LongColor}{\vbox{\hrule width \the\textwidth height 0.5pt} \textcolor{black}{#1}\ \vbox{\hrule width 40em height 0.5pt}}}; }}}\providecommand\AddOldversionToMargin[1]{\AddToHook {shipout/background}{\tikz[overlay,remember picture]{\node at ($(current page.west)+(2.5,0)$) [rotate=90] {\textcolor{\OldColor}{\vbox{\hrule width \the\textwidth height 0.5pt} \textcolor{black}{#1}\ \vbox{\hrule width 40em height 0.5pt}}}; }}}\providecommand\AddLineToMargin[3]{\AddToHook {shipout/background}{\tikz[overlay,remember picture]{\node at ($(current page.west)+(#2,0)$) [rotate=90] {\textcolor{#1}{\vbox{\hrule width \the\textwidth height 0.5pt} \textcolor{black}{#3}\ \vbox{\hrule width 40em height 0.5pt}}}; }}}\let\oldtocsection=\tocsection \let\oldtocsubsection=\tocsubsection \let\oldtocsubsubsection=\tocsubsubsection \renewcommand{\tocsection}[2]{\hspace{0em}\vspace*{0.1em}\oldtocsection{#1}{#2}}\renewcommand{\tocsubsection}[2]{\hspace{4ex}\oldtocsubsection{#1}{#2}}\renewcommand{\tocsubsubsection}[2]{\hspace{6ex}\oldtocsubsubsection{#1}{#2}}\usepackage{ulem}\usepackage{fancybox}\ifpdf \usepackage[pdftex]{lscape}\else \usepackage{lscape}\fi \makeatletter \newcommand{\verbatimfont}[1]{\def\verbatim@font{#1}}\makeatother \IfFileExists{mathabx.sty}{}{}\usepackage{amsfonts}\usepackage{amssymb}\usepackage{stmaryrd}\usepackage{amsmath}\usepackage{amsthm}\usepackage{dsfont}\usepackage{mathrsfs}\usepackage{twcal}\usepackage{accents}\usepackage[T1]{fontenc}\usepackage[latin1]{inputenc}\catcode`\=13 \def{+}\newcommand\assigncharacter[1]{\expandafter\newcommand\csname #1\endcsname{\mathds{#1}}}\FunctionForEach{,}{\assigncharacter}{A,B,C,D,E,F,G,I,J,K,M,N,Q,R,T,U,V,W,X,Y,Z}\renewcommand\assigncharacter[1]{\expandafter\newcommand\csname C#1\endcsname{\mathcal{#1}}}\FunctionForEach{,}{\assigncharacter}{A,B,C,D,E,F,G,H,I,J,K,L,M,N,O,P,Q,R,S,T,U,V,W,X,Y,Z}\renewcommand\assigncharacter[1]{\expandafter\newcommand\csname D#1\endcsname{\mathfrak{#1}}}\FunctionForEach{,}{\assigncharacter}{a,b,c,d,e,f,g,h,i,j,k,l,m,n,o,p,q,r,s,t,u,v,w,x,y,z,A,B,C,D,E,F,G,I,K,L,M,N,O,P,Q,R,S,T,U,V,W,X,Y,Z} \renewcommand\assigncharacter[1]{\expandafter\newcommand\csname S#1\endcsname{\mathscr{#1}}}\FunctionForEach{,}{\assigncharacter}{A,B,C,D,E,F,G,H,I,J,K,L,M,N,O,P,Q,R,T,U,V,W,X,Y,Z}\def\NewFont#1#2#3#4#5{\expandafter\font\csname #1display\endcsname =#1 at #2 \expandafter\font\csname #1normal\endcsname =#1 at #3 \expandafter\font\csname #1script\endcsname =#1 at #4 \expandafter\font\csname #1scriptscript\endcsname =#1 at #5 }\def\NewFontLetter#1#2{{\mathchoice {{\expandafter\hbox{\csname #1display\endcsname\char"#2}}}{{\expandafter\hbox{\csname #1normal\endcsname\char"#2}}}{{\expandafter\hbox{\csname #1script\endcsname\char"#2}}}{{\expandafter\hbox{\csname #1scriptscript\endcsname\char"#2}}}}}\NewFont{pxsyc}{9.00pt}{8.00pt}{7.00pt}{6.00pt}\NewFont{pxsya}{9.00pt}{8.00pt}{7.00pt}{6.00pt}\renewcommand{\rightsquigarrow}{\NewFontLetter{pxsya}{20}}\NewFont{p1xr}{10.00pt}{9.00pt}{8.00pt}{7.00pt}\NewFont{MnSymbolA5}{10.00pt}{9.00pt}{8.00pt}{7.00pt}\NewFont{MnSymbolC10}{10.00pt}{9.00pt}{8.00pt}{7.00pt}\NewFont{MnSymbolD10}{12.00pt}{11.00pt}{10.00pt}{9.00pt}\NewFont{MnSymbolF10}{12.00pt}{11.00pt}{10.00pt}{9.00pt}\renewcommand{\bigcap}{\mathop{\NewFontLetter{MnSymbolF10}{1D}}}\def\IndependenceX#1#2{#1\setbox0=\hbox{$#1x$}\kern\wd0\hbox to 0pt{\hss$#1\mid$\hss}\lower.9\ht0\hbox to 0pt{\hss$#1\smile$\hss}\kern\wd0}\def\nIndependenceX#1#2{#1\setbox0=\hbox{$#1x$}\kern\wd0 \hbox to 0pt{\mathchardef\nn=12854\hss$#1\nn$\kern1.4\wd0\hss}\hbox to 0pt{\hss$#1\mid$\hss}\lower.9\ht0 \hbox to 0pt{\hss$#1\smile$\hss}\kern\wd0}\NewFont{manfnt}{12.00pt}{11.00pt}{10.00pt}{9.00pt}\NewFont{favmr7y}{12.00pt}{11.00pt}{10.00pt}{9.00pt}\ifpdf \usepackage[pdftex,usenames,x11names]{xcolor}\else \usepackage[dvips,usenames,x11names]{xcolor}\fi \StackCreate{ColoR}{;}{?}\AtBeginDocument{\colorlet{defaultcolor}{.}\ColoRPush{defaultcolor}}\definecolor{Green}{rgb}{0.00,0.50,0.00}\definecolor{DarkGreen}{rgb}{0.00,0.40,0.00}\definecolor{gray}{rgb}{0.40,0.40,0.40} \renewcommand\textcolor[2]{\ColoRPush{#1}\color{\ColoRTop}#2\ColoRPop\color{\ColoRTop}}\usepackage[pdftex]{graphicx}\usepackage[all]{xy}\ifdefined\loadTIKZ \loadTIKZ \def\TIKZlabel#1{}\else \usepackage{tikz}\usetikzlibrary{matrix,arrows,calc,cd,decorations.pathmorphing,intersections}\fi \newcommand {\notion}[2][]{\def\temp{#1}\ifmmode #2 \ifx \temp\empty \index{$#2$}\else \index{$#1$}\fi \else {\bf #2}\ifx \temp\empty \index{#2}\else \index{#1}\fi \fi }\newcommand\NOPAGENUMBER[1]{}\usepackage{xr-hyper}\newcommand{\refX}[2]{\IfBeginWith{#1}{:}{\ref{\GetAfterColonAux#1;-#2}}{\cite[\ref{#1-#2}]{#1}}}\providecommand\hyperrefBOOKMARKS{true}\providecommand\hyperrefOPTIONS{destlabel=false}\ifdefined\BibLaTeX \usepackage[\hyperrefOPTIONS,pdftex,linktocpage,breaklinks,bookmarks=\hyperrefBOOKMARKS]{hyperref}\else \usepackage[\hyperrefOPTIONS,pdftex,linktocpage,pagebackref,breaklinks,bookmarks=\hyperrefBOOKMARKS]{hyperref}\fi \hypersetup{bookmarksdepth=3, colorlinks=true,allcolors=Green, linkcolor=DarkGreen, citecolor=violet, urlcolor=blue, runcolor=red, filecolor=magenta }\providecommand\url[1]{}\providecommand\nolinkurl[1]{}\providecommand\href[3][]{}\providecommand\hyperlink[2]{}\providecommand\hypertarget[2]{}\providecommand\hyperdef[3]{}\providecommand\hyperref[2][]{} \providecommand\hypersetup[1]{}\providecommand\pdfbookmark[3][]{}\providecommand\currentpdfbookmark[2]{}\providecommand\belowpdfbookmark[2]{}\providecommand\texorpdfstring[2]{} \newcommand\refDefined[3]{\ifcsname r@#1\endcsname #2 \else #3 \fi }\def\UndefinedRef#1{\large\bfseries\color{red} ??#1??}\def\UndefinedCite#1{\large\bfseries\color{magenta} ??#1??}\makeatletter \def\@setref#1#2#3{\ifx#1\relax \protect\G@refundefinedtrue \nfss@text{\reset@font\UndefinedRef{#3}}\@latex@warning{Reference `#3' on page \thepage \space undefined }\else \expandafter\Hy@setref@link#1\@empty\@empty\@empty\@nil{#2}\fi }\ifdefined\BibLaTeX \def\abx@missing@entry#1{\UndefinedCite{#1}}\else \def\@citex[#1]#2{\let\@citea\@empty \@cite{\@for\@citeb:=#2\do{\@citea \def\@citea{,\penalty\@m\ }\edef\@citeb{\expandafter\@firstofone\@citeb}\if@filesw \immediate\write\@auxout{\string\citation{\@citeb}}\fi \@ifundefined{b@\@citeb\@extra@b@citeb}{\mbox{\reset@font\UndefinedCite{\@citeb}}\G@refundefinedtrue \@latex@warning{Citation `\@citeb' on page \thepage \space undefined }}{\hbox{\csname b@\@citeb\@extra@b@citeb\endcsname}}}}{#1}}\fi \makeatother \fi \newcommand\pr{\begin{proof}}\def\ende{\end{proof}}\newtheoremstyle{LayoutVoid}{1ex}{0ex}{\normalfont}{}{\bf}{.}{1ex}{}\newcommand\stressstatement[1]{#1}\theoremstyle{plain}\swapnumbers \newcommand\maketheorem[1]{\newtheorem{#1}[theorem]{\stressstatement{#1}} \newtheorem{#1Definition}[theorem]{\stressstatement{#1 and Definition}}  }\FunctionForEach{,}{\maketheorem}{Conclusion,Conjecture,Corollary,Fact,Facts,Lemma,Observation,Observations,Proposition,Reminder,Scholium,Summary,Theorem}\theoremstyle{definition}\theoremstyle{remark}\FunctionForEach{,}{\maketheorem}{Convention,Counterexample,Counterexamples,Discussion,Example,Examples, Exercise,Exercises,Explanation,Notation,Project,Projects,Question,Questions,Remark,Remarks,Strategy,Warning}\theoremstyle{LayoutVoid}\numberwithin{equation}{section}\newcommand{\labelon}[1]{\marginpar{#1}}\newcommand{\labelx}[1]{{\def\temp{#1}\ifx\temp\empty\else \label{#1}\labelon{#1}\fi}}\def\GetAfterColon#1:#2;;{#2}\def\GetAfterPlus#1+#2;;{#2}\newenvironment{FACT}[2]{\edef\currentlabel{#2}\IfBeginWith{#1}{:}{\def\tempFactName{void}\def\tempFreeTitle{\GetAfterColon#1;;\ }}{\IfBeginWith{#1}{+}{\def\tempFactName{voidTheorem}\def\tempFreeTitle{\GetAfterPlus#1;;\ }}{\def\tempFactName{#1}\def\tempFreeTitle{}}}\def\tempfacT{\end{\tempFactName}}\begin{\tempFactName}\labelx{#2}\textup{\textbf{\tempFreeTitle}}\capitalize[q]{#1}\caselower[q]{#1}}{\tempfacT}\newcommand{\inv}{\mathrm{inv}}\newcommand{\con}{\mathrm{con}}\catcode95=12 \catcode95=8 \newcommand\bdl{{\ifmmode \mathrm{bdlat}\else {bounded distributive lattice}\fi}}\let\disjoint\cupdot \newcommand{\st}{{\ \vert \ }}\let\temp\phi \let\phi\varphi \let \varphi\temp \let\temp\theta \let\theta\vartheta \let \vartheta\temp \let\eps\varepsilon  \let\0\emptyset \newcommand{\into}{\hookrightarrow}\makeatletter \newcommand{\xRightarrow}[2][]{\ext@arrow 0359\Rightarrowfill@{#1}{#2}}\newcommand{\xLeftarrow}[2][]{\ext@arrow 0359\Leftarrowfill@{#1}{#2}}\newcommand{\xonto}[2][]{\ext@arrow 0359\rightarrowfill@ {#1}{#2}\mathrel{\mspace{-15mu}}\rightarrow}\newcommand{\xinto}[2][]{\lhook\joinrel\ext@arrow 0359\rightarrowfill@ {#1}{#2}}\makeatother \newcommand{\lra}{\longrightarrow}\newcommand{\Ra}{\Rightarrow}\newcommand{\spez}{\,\rightsquigarrow\,}\newcommand{\mal}{{\cdot}}\newcommand\qc{quasi-compact} \newcommand{\addressTressl}{The University of Manchester, Department of Mathematics, Oxford Road, Manchester M13 9PL, UK}\newcommand{\emailTressl}{marcus.tressl@manchester.ac.uk}\newcommand{\homepageTressl}{\url{http://personalpages.manchester.ac.uk/staff/Marcus.Tressl/}}\newcommand{\monthname}[1]{\ifcase#1 \or January \or February \or March \or April \or May \or June \or July \or August \or September \or October \or November \or December \fi}\newcommand\LongColor{teal}\newcommand\OldColor{gray}\newcommand\COL{\ifmmode\colon\else :\ \fi}\newcommand\kat[1]{{\tt #1}} \newcommand\operator[1]{\mathop{\operatorname{#1}}\nolimits}\newcommand\kon{{\mathcal{K}}}\newcommand\qcop{\mathop {\mathring{\kon}}\nolimits}\newcommand\qccl{\mathop {\overline {\kon }}\nolimits} \newcommand{\id}{\operator{id}}\newcommand{\powerset}{\operator{\DP}}\newcommand{\Spec}{\operator{Spec}}\newcommand{\zSpec}{\operator{z\text{-}Spec}}\newcommand{\Spez}{\operator{Spez}}\newcommand{\PrimI}{\operator{PrimI}}\newcommand{\Gen}{\operator{Gen}}\newcommand{\interior}{\operator{int}}\newcommand{\Ann}{\operator{Ann}}\newcommand{\Coz}{\operator{Coz}}\IfFileExists{C:/wb/System64/WinBatch.exe}{}{}\ifdefined\isinput\endinput\else\fi \usepackage{todonotes}\def\UndefinedCite#1{\large\bfseries\color{blue} ??#1??}\IfFileExists{DST.aux}{\externaldocument[DST-]{DST}\hypersetup{bookmarksdepth=3, colorlinks=true,allcolors=Green, linkcolor=DarkGreen, citecolor=violet, urlcolor=blue, runcolor=red, filecolor=black }}{\externaldocument[DST-]{/TEX/math/SpectralSpaces/DST}}\renewcommand{\labelon}[1]{}\newif\ifprivate\privatefalse \newif\iflongversion\longversionfalse \renewcommand\LongColor{teal}\newif\ifoldversion\oldversionfalse \ifprivate \AddPrivateToMargin{private version} \fi \iflongversion \AddLongversionToMargin{long version} \fi \ifoldversion \AddOldversionToMargin{old version included} \fi \newcommand\DST[2][]{\refDefined{DST-#2}{\cite[{\ref{DST-#2}#1}]{DiScTr2019}}{\ref{#2}}}\newcommand\GilJer[2]{\cite[{{\href[page=#1+9]{\DIRroot../ARTICLE/RAAG/C(X)/Gillman,Jerison - Rings of Continuous Functions.pdf}{#2}}}]{GilJer1960}}
\begin{document} \title[]{Pseudocomplementation in rings of continuous functions} \author{Guram Bezhanishvili} \address{Department of Mathematical Sciences, New Mexico State University, NM, USA\newline Homepage: \url{https://sites.google.com/view/guram-bezhanishvili/}} \email{guram@nmsu.edu} \author{Marcus Tressl} \address{\addressTressl\newline Homepage: \homepageTressl} \email{\emailTressl} \date{\today}  \begin{abstract} We study rings of real-valued continuous functions in terms of pseudocomplementation conditions on various lattices attached to their prime spectrum. We fully characterize pseudocomplementation in all cases and have an almost complete characterization of relative pseudocomplementation. \end{abstract} \maketitle \tableofcontents \section{Introduction} \noindent The ring $C(T)$ of continuous real-valued functions on a topological space $T$ is one of the most studied objects in topology \cite{GilJer1960}. It is a classic result of Gel'fand and Kolmogorov \cite[Theorem~IV']{GelKol1939} based on earlier work of Stone \cite{Stone1937} that the Stone-\v{C}ech compactification of a completely regular space $X$ can be described as the maximal spectrum of $C(T)$ and of its subring $C^*(T)$ of bounded functions. This has generated a comprehensive study of maximal spectra of the rings $C(T)$, as well as their minimal spectra, see for instance \cite{GilJer1960,HenWoo2004}. However, the complicated structure of full prime spectra $\Spec C(T)$ and $\Spec C^*(T)$ (as for instance in \cite{Schwar1997,Tressl2007}) is investigated to a lesser degree. \par We are contributing to this study in the following way. By Stone duality, prime spectra of rings are equivalently described by the lattice of their compact and open subsets. We aim to provide characterizations of variants of pseudocomplementation properties of these lattices and their order duals in terms of algebraic properties of $C(T)$ and topological properties of $\Spec C(T)$. For instance, we have complete characterizations for pseudocomplementation and an almost complete characterization of relative pseudocomplementation. In order to carry these out, we recall and add to the knowledge of pseudocomplementation in lattice theory. This is done in sections \ref{sectionPC} and \ref{sectionTreeRoot}, which we describe now. \par \smallskip The study of pseudocomplementation and relative pseudocomplementation has been a mainstream topic in lattice theory \cite{BalDwi1974,Gratze2011}, which has important consequences in algebraic logic as such structures serve as algebraic models of various non-classical logics \cite{RasSik1963,Rasiow1974}. In particular, there is a well-developed duality theory for such structures, which utilizes Stone duality \cite{Stone1937} and Priestley duality \cite{Priest1970} for bounded distributive lattices, the two being two sides of the same coin \cite{Cornis1975}. \par From Hochster's characterization of prime spectra of commutative rings \cite{Hochst1969}, it follows that spectra of both bounded distributive lattices and of commutative rings describe the same class of topological spaces, the so-called {\it spectral spaces} (see \cite{DiScTr2019} for a detailed account). \par There are various characterizations of when a bounded distributive lattice $L$ is pseudocomplemented or a Stone algebra using either the spectral space of $L$ or its sister Priestley space (see \cite{Gratze1963,CheGra1969,Priest1974,Priest1975}), as well as when it is a Heyting algebra or a coHeyting algebra (see \cite{Esakia1974,Esakia2019,BBGK2010,DiScTr2019}). \par \smallskip\noindent In this paper we utilize the above results, as well as the structure of $\Spec C(T)$, to obtain several characterizations of when the lattice $\qcop(\Spec C(T))$ of compact open sets of $\Spec C(T)$ is pseudocomplemented or a coHeyting algebra. Our main results include the following characterizations: \par \smallskip\noindent Let $T$ be a completely regular space. Then \begin{itemize} \item $\qcop(\Spec C(T))$ is pseudocomplemented iff $X$ is basically disconnected iff $\qcop(\Spec C(T))$ is a Stone algebra (Theorem \ref{CofTSemiHeyting}). As a consequence we also show that $\qcop(\Spec C(T))$ is pseudocomplemented iff $\qcop(\Spec C^*(T))$ is pseudocomplemented (Corollary \ref{SpecCbetaXSemiH}) \item If $T$ is metrizable, then $\qcop(\Spec C(T))$ is a Heyting algebra iff $X$ is discrete (Corollary \ref{semiHeytingMetric}). \item In Theorem \ref{normalXcoHeytingCT} we show that the order dual of $\qcop(\Spec C(T))$ is pseudocomplemented iff it is a Heyting algebra iff it is a Stone algebra iff $X$ is a P-space. Furthermore, all these conditions on $\Spec C(T)$ are equivalent to their formulations for the subspace of z-prime ideals $\zSpec C(T)$ (see \ref{ContF1} for definitions). \end{itemize} \par \smallskip\noindent On the way we will also see how known results fit into the topic of pseudocomplemntation. For instance, on the ring theoretic side we reframe known characterizations of basically disconnected spaces (\ref{CofTSemiHeyting}) and Baer rings (\ref{QCOPStoneLatticeApplyStone}); on the lattice theoretic side, Theorem \ref{QCOPStoneLattice} combines and extends various known characterizations of Stone algebras (see \ref{QCOPStoneLatticeApply}). \section{Preliminaries} \noindent In this preliminary section we briefly recall Stone and Priestley dualities for bounded distributive lattices. The former gives rise to spectral spaces, while the latter gives rise to Priestley spaces; we briefly discuss the Cornish isomorphism between the two categories. We also recall prime spectra of commutative rings. Our notation mostly follows that in \cite{DiScTr2019}; for the reader interested in the frame theoretic approach to the matter we refer to \cite{Johnst1986}. \begin{FACT}{:Spectral Spaces}{SpSp}\quad Let $X$ be an arbitrary topological space. For $x,y\in X$ we say that $x$ \notion[]{specializes} to $y$, and write $x\spez y$, if $y\in \overline{\{x\}}$. It is well known (and easy to check) that the specialization relation $\spez$ is always a preorder, and that it is a partial order iff $X$ is a $T_0$-space. For $S \subseteq X$ we call \[ \Spez(S)=\{x\in X\st x\in \overline{\{s\}}\text{ for some }s\in S\} \] the set of \textbf{specializations of $S$} (in $X$) and \[ \Gen(S)=\{x\in X\st s\in \overline{\{x\}}\text{ for some }s\in S\} \] the set of \textbf{generalizations of $S$} (in $X$). Define \begin{align*} X^{\max}=\{x\in X\st \overline{\{x\}}=\{x\}\}\quad\text{and}\quad X^{\min}=\{x\in X\st \Gen(x)=\{x\}\}. \end{align*} If $X$ is the prime spectrum of a ring, then $X^{\max}$ is the set of maximal ideals of the ring. Since we are primarily interested in this example, we read $\spez $ as $\leq$ (instead of the alternative reading as $\geq$). Therefore, $\Spez(S) = {\uparrow} S$ and $\Gen(S) = {\downarrow} S$. We set \begin{align*} \CO(X)&=\text{open subsets of }X,\cr \qcop(X)&=\{U\in\CO(X)\st U\text{ is compact}\},\cr \qccl(X)&=\{X\setminus U\st U\in\qcop(X)\}, \cr \CK(X)&=\text{ the Boolean algebra generated by }\qcop(X)\text{ in the powerset of }X. \end{align*} \par \noindent Recall that $X$ is a \notion[]{spectral space} if it is sober and $\qcop(X)$ is a bounded sublattice and a basis for $\CO(X)$. For a spectral space $X$, we write $X_\inv$ for the \notion[]{inverse space} of $X$, which coincides with the de Groot dual of $X$ (see, e.g., \DST{defnInverseTopology}). Therefore, $\qcop(X_\inv)=\qccl(X)$. Furthermore, we write $X_\con$ for the \textbf{patch space} of $X$ (see, e.g., \DST{defnPatchSpace}). Thus, $X_\con$ is a \textbf{Boolean space} (compact, Hausdorff, and zero-dimensional) and $\qcop(X_\con)=\CK(X)$. Properties of these topologies will be referred to with the respective decoration, like `inversely compact' (meaning `compact in $X_\inv$' ) or `patch closed' (meaning `closed in $X_\con$'). \par The {\bf Stone representation} of \bdl s states that each such lattice $L$ is isomorphic to the lattice $\qcop(X)$ for some spectral space $X$, which is unique up to a homeomorphism. There are various ways to see this (see, e.g., \cite[sections 3.1--3.3]{DiScTr2019}). We choose $X$ to be the set \[ X=\PrimI(L)=\{\Dp\subseteq L\st \Dp\text{ is a prime ideal of }L\}, \] whose topology is given by the basis consisting of the sets $D(a):=\{\Dp\in X\st a\notin\Dp\}$ for $a\in L$; we write $V(a):=\{\Dp\in X\st a\in\Dp\}$. Then $X$ is a spectral space and the map \( L\lra \qcop(X), \quad a\mapsto D(a) \) is a lattice isomorphism. The definition of the space $\PrimI(L)$ looks formally very similar to the definition of the prime spectrum of a ring; as the latter is our target, we chose this presentation. By definition we see that \[ \Dp\spez \Dq \Longleftrightarrow \Dp\subseteq \Dq. \] Thus, $X^{\max}$ is the space of maximal ideals of $L$. The representation then gives rise to lattice isomorphisms \[ L_\inv\lra \qcop(X_\inv)=\qccl(X),\quad A\lra \qcop(X_\con)=\CK(X), \] where $L_\inv$ is the order dual of $L$, and $A$ is the Boolean envelope of $L$ (see, e.g., \DST{DefBooleanEnvelope}). \par \smallskip\noindent The morphisms in the category \kat{Spec} of spectral spaces are \notion{spectral maps}, i.e. (continuous) maps $f:X\lra Y$ with the property $V\in \qcop(Y)\Ra f^{-1}(V)\in\qcop(X)$. Stone duality says that \kat{Spec} is anti-equivalent (aka dually equivalent) to the category of \bdl s with bounded lattice morphisms. \end{FACT} \begin{FACT}{:Priestley spaces}{}\quad A {\bf Priestley space} is a Boolean space $X$ equipped with a partial order $\le$ that satisfies the {\bf Priestley separation axiom}: \[ x \not\le y \Longrightarrow \exists D \mbox{ clopen down-set} : x \notin D \mbox{ and } y \in D. \] The {\bf Priestley representation} then states that each bounded distributive lattice $L$ is isomorphic to the lattice of clopen down-sets of some Priestley space, which is unique up to an order-homeomorphism. In particular, if $X$ is the spectral space of $L$, then $(X_\con,\spez)$ is the Priestley space of $L$, where $X_\con$ is the patch space of $X$ and $\spez$ is the specialization order. \par \smallskip\noindent The morphisms in the category \kat{Priestley} of Priestley spaces are continuous and order preserving maps and Priestley duality says that \kat{Priestley} is anti-equivalent to the category of \bdl s with bounded lattice morphisms. \par \smallskip\noindent As is evident, there is a close connection between spectral spaces and Priestley spaces. Indeed, each spectral space $X$ gives rise to the Priestley space $(X_\con,\spez)$, and the topology of the spectral space is recovered as the topology of open down-sets of $(X_\con,\spez)$. As was demonstrated by Cornish in \cite{Cornis1975}, this correspondence extends to an isomorphism of \kat{Spec} and \kat{Priestley}, see also \cite{BBGK2010} and \DST{PriestleyCATisomorphism}. \end{FACT} \medskip\noindent Here is a dictionary connecting terminology from spectral spaces and Priestley spaces: Let $X$ be a spectral space and let $P=(X_\con,\leq)$ be the corresponding Priestley space; recall that $x\leq y\iff y\in \overline{\{x\}}$. \begin{enumerate} \item The patch closed subsets of $X$ (also called \textit{proconstructible} in \cite{DiScTr2019}) are the closed subsets of $P$. \item The closure in $X$ of a patch closed set $S$ is ${\uparrow}S$. \item The closure in the inverse topology of a patch closed set $S$ is ${\downarrow}S$. \item The elements of $\CK(X)$ are called \textbf{constructible sets} and they are precisely the clopen subsets of $P$. \item The open subsets of $X$ are the open down-sets of $P$. If $U\subseteq X$, then $U$ is compact open iff $U$ is open and constructible iff $U$ is a clopen down-set of $P$. \item A set $S$ is compact in $X$ if and only if ${\downarrow}S$ is patch closed. \end{enumerate} \begin{FACT}{:Prime spectra of commutative rings}{}\quad We adopt the notational conventions of \cite[section 12]{DiScTr2019}, which also points to general texts on the prime spectrum of a ring and which might be used for additional elementary facts about these. Let $A$ be a ring, which always means commutative and unital in this paper. Recall that the prime spectrum (or Zariski spectrum) of $A$ is the spectral space $\Spec(A)$ of prime ideals of $A$ with the topology having the sets $D(f):=\{\Dp\in \Spec(A)\st f\notin \Dp\}$ as a basis; we write $V(f):=\{\Dp\in \Spec(A)\st f\in \Dp\}$. Specialization in $\Spec(A)$ is inclusion and in the setup of this paper this means $\Dp\leq \Dq\iff \Dp\subseteq \Dq$ ($\iff \Dq\in \overline{\{\Dp\}}$). The closed sets of $\Spec(A)$ are those of the form $V(S):=\{\Dp\in \Spec(A)\st S\subseteq \Dp\}$ for $S\subseteq A$. There is an antitone Galois connection between subsets of $A$ and subsets of $\Spec(A)$ mapping $S\subseteq A$ to $V(S)$ and $Z\subseteq \Spec(A)$ to $\bigcap Z$; this defines a bijection between radical ideals of $A$ and closed subsets of $\Spec(A)$. \end{FACT} \begin{FACT}{:Prime spectra of rings of continuous functions}{ContF1}\quad Let $T$ be a completely regular space. We write $C(T)$ for the ring of continuous functions $T\lra \R$. Note that $C(T)$ is also a poset where $f\leq g$ means $\forall t\in T: f(t)\leq g(t)$. We adopt the terminology from \cite{GilJer1960}. \begin{enumerate}[{\rm (i)},itemsep=1ex] \item\labelx{ContFidealZideal} An ideal $I$ of $C(T)$ is called \textbf{z-radical} or a \textbf{z-ideal} if $f\in I$ and $Z(g)\supseteq Z(f)$ implies $g\in I$. Here $Z(f)=\{x\in T\st f(x)=0\}$ is the \textbf{zero set} of $f$. The complements are called \textbf{cozero sets} and they form the bounded sublattice $\Coz(T)$ of the frame $\CO(T)$ of opens (use $Z(f)\cap Z(g)=Z(f^2+g^2)$). \item\labelx{ContFSpecIsRootSystem} The prime ideal spectrum $\Spec C(T)$ of $C(T)$ is a spectral root system, see \DST{charSpezInRealSpec}. \item\labelx{ContFFspace} $T$ is called an \notion[]{F-space} if $\Spec C(T)$ is \textbf{stranded}, i.e. $\Spec C(T)$ is also a forest (see \GilJer{208}{14.25}) and a \notion[]{P-space} if $\Spec C(T)$ is Boolean, see section \ref{sectionTreeRoot} for details. \item\labelx{ContFBasicallyDisconnected} A completely regular space $T$ is called \notion[basically disconnected]{basically disconnected} if the closure of any cozero set is open. Such a space is zero-dimensional (i.e. the clopen sets form a basis) because when $x\in U\in \CO(T)$, then there is $V\in\Coz(T)$ with $x\in V$ and $\overline{V}\subseteq U$. \item\labelx{ContFeps} The map $\eps:T\lra \Spec C(T)$ that sends $x\in T$ to $\Dm_x:=\{f\st f(x)=0\}$ is an embedding of topological spaces and its image is contained in the subspace $\beta T$ of maximal ideals of $C(T)$. The space $\beta T$ is a compact Hausdorff space and is called the {\bf Stone--\v Cech Compactification} of $T$ (cf. \DST{defnBetaX}). \item\labelx{ContFz} The prime z-ideals of $\Spec C(T)$ form a spectral subspace of $\Spec C(T)$, denoted by $\zSpec C(T)$ and this is the Stone dual of the lattice $\Coz(T)$. Hence $\Coz(T)\cong\qcop(\zSpec C(T))$ and $\qccl(\zSpec C(T))$ is (isomorphic to) the order dual of $\Coz(T)$. Each maximal ideal and every minimal prime ideal of $C(T)$ is z-radical. The space $\zSpec C(T)$ is the patch closure of $\eps(T)$. All this is more or less spread out in \cite{GilJer1960} and in \cite{Schwar1997}; we refer to \cite[p. 145]{Tressl2006} for a summary with more details. \end{enumerate} \par \noindent The spaces considered above are displayed in the following diagram of spaces and subspaces, where we consider $\eps$ as inclusion and write $Z=\zSpec C(T)$: \begin{center} \begin{tikzcd}[row sep=scriptsize,column sep=7ex,ampersand replacement=\&] T \ar[r,hook] \& \beta T = X^{\max} \ar[r,hook] \& Z=\overline{T}^\con \ar[r,hook] \& X = \Spec C(T)\\ \ \& X^{\min}\ar[ru,hook] \& \ \& \ \end{tikzcd} \end{center} \noindent Here $\overline{T}^\con$ denotes the closure of $T$ for the patch topology. Taking preimages under the embedding $T\lra Z$ induces a lattice isomorphism $\qcop(Z)\lra \Coz(T)$. \par \end{FACT} \section{Pseudocomplementation in bounded distributive lattices} \labelx{sectionPC} \noindent In this section we focus on the lattice theoretic side of pseudocomplementation and variants, as well as their translation into topology using Stone duality. A particular emphasis is given to Heyting algebras and Stone algebras. The results will be deployed in section \ref{sectionRings} to characterize pseudocomplementation of various lattices attached to rings of continuous functions. \begin{FACT}{:Pseudocomplementation}{pscompBasic} Let $L=(L,\land,\lor,\perp,\top)$ be a \bdl. We recall that the {\bf pseudocomplement} of $a\in L$ is the largest element of the set \[ \{ x \in L \mid a\land x = \perp \}. \] We write $a^*$ for the pseudocomplement when it exists. If all elements of $L$ have a pseudocomplement, then $L$ is called {\bf pseudocomplemented}.\footnote{If the pseudocomplementation map is named, the resulting structure is a {\it p-algebra} (see, e.g., \cite[section 7.1]{Blyth2005}).} The following theorem provides a characterization of the duals of pseudocomplemented lattices: \end{FACT} \begin{FACT}{Theorem}{pscompBasic} Let $L$ be a bounded distributive lattice and $X$ its space of prime ideals. The following conditions are equivalent: \begin{enumerate}[{\rm (i)}] \item $L$ is pseudocomplemented. \item For every $U\in\qcop(X)$, the closure $\overline{U}={\uparrow} U$ is constructible. \item\labelx{pscompBasicCharPC} The following two conditions hold.\footnote{Observe that neither of these conditions can be dropped, see \DST{XminCompactNotpsCompl}.} \begin{enumerate}[label=(\alph*)] \item\labelx{pscompBasicCharPCreg} For $U\in \qcop(X)$ the open regularization $\interior(\overline{U})$ belongs to $\qcop(X)$, and \item\labelx{pscompBasicCharPCcomp} $X^{\min}$ is \qc; in this case $X^{\min}$ is even a patch closed subset of $X$, see \DST{minimalPointsCompact}. \end{enumerate} \end{enumerate} \end{FACT} \begin{proof} The equivalence of (i) and (ii) is proved in \cite[Proposition 1]{Priest1974}. For the equivalence of (ii) and \ref{pscompBasicCharPC} see \DST{sHeytingIsQuarterHeytingPlusXminComp}. \end{proof} \par \noindent Following \cite{BeBeIl2016}, we call the spectral spaces satisfying the condition in the above theorem {\bf PC-spaces}; in \DST{DefnHeytingSpectral} they are called \textit{semi-Heyting}. We next recall that a pseudocomplemented lattice $L$ is a {\bf Stone algebra} according to \cite{GraSch1957} provided $a^* \vee a^{**} = \top$ for each $a \in L$. Let $X$ be a PC-space. For $U \in \qcop(X)$ we have $U^{**} = \interior(\overline{U})$ (see \DST{PseudoComplementInLocSp}), yielding the dual characterization given in \ref{QCOPStoneLattice} of Stone algebras (cf. \cite{GraSch1957,DumLem1959,Priest1974}). \begin{FACT}{Remark}{cozeroCompInSpec} Condition \ref{pscompBasicCharPCcomp} in \ref{pscompBasic} alone also has a meaning in terms of complementation, which goes as follows. In any spectral space $X$ the set $X^{\min}$ is compact if and only if for all $U\in \qcop(X)$ there is some $V\in \qcop(X)$ such that $X^{\min}=U^{\min}\disjoint V^{\min}$. This follows from \DST{minimalPointsCompact}, which says that $X^{\min}$ is compact iff $X$ and $X_\inv$ induce the same topology on $X^{\min}$. \par The property $X^{\min}=U^{\min}\disjoint V^{\min}$ resembles the property of $V$ being a ``generic complement'' of $U$ in $X$ and it is equivalent to saying that $U\cap V=\0$ and $U\cup V$ is dense in $X$. \end{FACT} \begin{FACT}{:Normal spectral spaces and retraction maps}{NormalSpec} Normal spectral spaces are discussed in detail in \DST{sectionNormal}. A spectral space is normal (in the sense of topology) if and only if the closure of every point contains a unique closed point. Another way of saying this is that the inclusion map $X^{\max}\into X$ possesses a retraction $r:X\lra X^{\max}$ that preserves specialization. \par In such a space $X$, the map $r:X\lra X^{\max}$ that sends $x\in X$ to the unique closed point in $\overline{\{x\}}$ is a continuous, closed, and proper retraction of the embedding $X^{\max}\into X$ \cite[Proposition 3, p.~230]{CarCos1983}.\footnote{A ring whose spectrum is normal is called a \textit{Gel'fand ring} (aka a \textit{PM-ring}).} \par Now note that for a subset $S$ of $X$ we have $\Gen(\Spez(S))={\downarrow}{\uparrow}S=r^{-1}(r(S))$. Hence if $S$ is closed, then ${\downarrow}S={\downarrow}{\uparrow}S$ is closed as well. \end{FACT} \smallskip\noindent For an arbitrary spectral space $X$ we now characterize the existence of a continuous retraction $X\lra X^{\min}$ of the inclusion map $X^{\min}\into X$; notice that this condition is \textbf{not} equivalent to saying that $X_\inv$ is normal. \par \belowpdfbookmark{Gr\"atzer-Schmidt}{BookmarkQCOPStoneLattice} \begin{FACT}{Theorem}{QCOPStoneLattice} Let $L$ be a \bdl\ and $X$ its space of prime ideals. The following conditions are equivalent: \looseness=-1 \begin{enumerate}[{\rm (i)},itemsep=1ex] \item\labelx{QCOPStoneLatticeStone} $L$ is a Stone algebra; recall that $L$ is isomorphic to $\qcop(X)$. \item\labelx{QCOPStoneLatticeBasicDisconnected} For all $U\in \qcop(X)$ the closure $\overline{U}={\uparrow}U$ is open. \item\labelx{QCOPStoneLatticeUpCommutes} $X^{\min}$ is patch closed and for all $x,y,z \in X$, if $y,z \le x$ then there is $u \in X$ with $u \le y,z$. In other words, ${\uparrow}(D \cap E) = {\uparrow} D \cap {\uparrow} E$ for any two down-sets $D,E$. \item\labelx{QCOPStoneLatticeNormProCon} $X^{\min}$ is patch closed and $X_\inv$ is normal, i.e. for each $x \in X$ there is a unique $y \in X^{\min}$ such that $y \le x$. \item\labelx{QCOPStoneLatticeProjection} $X_\inv$ is normal and the map $X\lra X$, sending $x\in X$ to the unique minimal point specializing to $x$, is spectral. \item\labelx{QCOPStoneLatticeRetraction} There is a continuous retraction $X\lra X^{\min}$ of the inclusion map $X^{\min}\into X$. \par \end{enumerate} \par \end{FACT} \begin{proof} \ref{QCOPStoneLatticeNormProCon}$\Leftrightarrow$\ref{QCOPStoneLatticeUpCommutes}: Both implications follow from the fact that $X={\uparrow}(X^{\min})$. \par \smallskip\noindent \ref{QCOPStoneLatticeNormProCon}$\Leftrightarrow$\ref{QCOPStoneLatticeProjection} holds by \DST{XmaxProConNormal} applied to $X_\inv$. \par \smallskip\noindent \ref{QCOPStoneLatticeProjection}$\Ra $\ref{QCOPStoneLatticeRetraction} is clear. \par \smallskip\noindent \ref{QCOPStoneLatticeRetraction}$\Ra $\ref{QCOPStoneLatticeBasicDisconnected}. By assumption, there is a continuous map $s:X\lra X^{\min}$ whose restriction to $X^{\min}$ is the identity map. Take $U\in\qcop(X)$. Since $U$ is a down-set, we know that ${\uparrow}U={\uparrow}(X^{\min}\cap U)$. In order to show that ${\uparrow}U$ is open, we deploy continuity of $s$, so it is sufficient to prove ${\uparrow}U=s^{-1}(X^{\min}\cap U)$. \par $\subseteq$. If $U\ni u\leq x$, then take $y\in X^{\min}\cap {\downarrow}u$. By continuity of $s$ this implies $s(y)\leq s(x)$ and then $y=s(x)$ because $s$ is the identity on $X^{\min}$ and $X^{\min}$ has no proper specializations. \par $\supseteq $. If $s(x)\in X^{\min}\cap U$, then take $y\in X^{\min}$ with $y\leq x$; we see again that $y=s(y)=s(x)\in U$, thus $x\in {\uparrow}U$. \par \medskip\noindent Hence we know that \ref{QCOPStoneLatticeNormProCon}$\Rightarrow$\ref{QCOPStoneLatticeUpCommutes}$\Rightarrow$\ref{QCOPStoneLatticeProjection}$\Rightarrow$\ref{QCOPStoneLatticeRetraction}$\Rightarrow $\ref{QCOPStoneLatticeBasicDisconnected}. But \ref{QCOPStoneLatticeBasicDisconnected} implies that $L$ is pseudocomplemented, using \ref{pscompBasic}. Therefore, for the remaining implications we may assume that $L$ is pseudocomplemented (and $X^{\min}$ is patch closed). \par Now we may refer to \cite{Priest1974} which under this assumption has already proved the equivalences of \ref{QCOPStoneLatticeStone},\ref{QCOPStoneLatticeBasicDisconnected}, \ref{QCOPStoneLatticeUpCommutes}, and \ref{QCOPStoneLatticeNormProCon}; see Propositions 2 and 3 in that paper. \end{proof} \begin{FACT}{:Remark}{QCOPStoneLatticeApply} \begin{enumerate}[{\rm (i)}] \item Gr\"atzer and Schmidt in \cite{GraSch1957} resolve G.~Birkhoff's problem no.~70 \cite[p.~149]{Birkho1948} by showing that a pseudocomplemented distributive lattice is a Stone algebra if and only if all distinct minimal prime ideals are coprime. The equivalence of \ref{QCOPStoneLatticeStone} and \ref{QCOPStoneLatticeNormProCon} in \ref{QCOPStoneLattice} implies this result. \item The equivalence of \ref{QCOPStoneLatticeStone} and \ref{QCOPStoneLatticeUpCommutes} in \ref{QCOPStoneLattice} also dates back to \cite{DumLem1959}, providing an alternate solution of Birkhoff's problem (see the paragraph after Theorem 3 in their paper). \item In \cite[Proposition~2.37]{ChaZak1997} posets satisfying the first order condition in \ref{QCOPStoneLattice}\ref{QCOPStoneLatticeUpCommutes} are called \textit{strongly directed}. This condition is also known as the \textit{confluence} or \textit{Church-Rosser} condition \cite[Example~3.44]{BldRVe2001} and plays an important role in Kripke semantics for modal logic. \end{enumerate} \end{FACT} \begin{FACT}{:Application}{QCOPStoneLatticeApplyStone} If $X$ is the prime spectrum of a reduced commutative ring $A$ and $U\in \qcop(X)$, then $U$ is of the form $D(f_1)\cup\ldots\cup D(f_n)$ for some $n\in \N$ and $f_i\in A$. Therefore, condition \ref{QCOPStoneLatticeBasicDisconnected} of \ref{QCOPStoneLattice} is equivalent to saying that for each $f\in A$ the closure of $D(f)$ is open. Now one checks without difficulty that the closure of $D(f)$ is $V(\Ann(f))$ and clopen subsets of $X$ are of the form $V(e\mal A)$ for some idempotent element $e\in A$. Since $\Ann(f)$ and $e\mal A$ are radical ideals in any reduced ring\footnote{If $h^2\in e\mal A$, then $h^2(1-e)^2=0$, so $h(1-e)=0$ and $h=eh\in e\mal A$.}, condition \ref{QCOPStoneLatticeBasicDisconnected} of \ref{QCOPStoneLattice} is equivalent to saying that for every $f\in A$ there is an idempotent element $e\in A$ with $\Ann(f)=e\mal A$. Reduced rings with this property are called \notion[Baer ring]{Baer rings} in \cite{Kist1974}, which we adopt here (other authors call such rings \textit{weak Baer}). \par Hence for the prime spectrum $X$ of a reduced commutative ring $A$, \ref{QCOPStoneLattice} implies that $A$ is a Baer ring if and only if $X^{\min}$ is a retract of $X$ (reproving \cite[Theorem 2]{Kist1974}), if and only if $\qcop(X)$ is a Stone algebra. \end{FACT} \medskip\noindent Theorem \ref{QCOPStoneLattice} can also be formulated in the language of the inverse topology (utilizing also Theorem \ref{pscompBasic}). We only record those statements that are used later on. \begin{FACT}{Corollary}{QCCLStoneLattice} For a spectral space $X$, the following are equivalent: \begin{enumerate}[{\rm (i)}] \item\labelx{QCCLStoneLatticeStone} $\qccl(X)$ is a Stone algebra. \item\labelx{QCCLStoneLatticeBasicDisconnected} For all $C\in \qccl(X)$ the closure ${\downarrow}C$ of $C$ in the inverse topology is clopen. \item\labelx{QCCLStoneLatticeNormalCoHeyting} $X$ is normal and $\qccl(X)$ is pseudocomplemented. \item\labelx{QCCLStoneLatticeNormProCon} $X$ is normal and $X^{\max}$ is patch closed. \par \end{enumerate} \end{FACT} \begin{FACT}{:Relative pseudocomplementation}{} Let $L$ be a bounded distributive lattice and $a,b \in L$. We recall that the {\bf relative pseudocomplement} of $a$ with respect to $b$ is the largest element of the set \( \{ x \in L \mid a \wedge x \le b \}, \) \cite{RasSik1963,BalDwi1974}, denoted by $a\to b$; another name is {\bf implication}. If $a \to b$ exists for all $a,b \in L$, then $L$ is called {\bf relatively pseudocomplemented } or a {\bf Heyting algebra}. The dual spaces of Heyting algebras are called \notion{Esakia spaces} (or \textit{Heyting spaces} in \cite[section 8.3]{DiScTr2019}). They were first described by Esakia \cite{Esakia1974} and further studied by numerous authors. The next theorem is well known (see \cite{Esakia1974,Esakia2019,BBGK2010,BeGaJi2013,DiScTr2019}). \end{FACT} \begin{FACT}{Theorem}{HeytingEasyII} Let $L$ be a bounded distributive lattice and $X$ its dual spectral space. The following are equivalent: \begin{enumerate}[{\rm (i)},itemsep=0ex] \item \labelx{HeytingEasyIIQCOPheyting} $L$ is a Heyting algebra. \item \labelx{HeytingEasyIISpaceHeyting} The closure of any constructible subset of $X$ is constructible; recall that the closure of any patch closed set is the upset that it generates. \item\labelx{HeytingEasyIISemiOnClosed} Every closed and constructible subspace of $X$ is a PC-space. \item\labelx{HeytingEasyIIInvClosed} For any $S\subseteq X$, the closure of $S$ for the inverse topology is the patch closure of ${\downarrow }S$. (Notice that in any spectral space, the closure of a subset is the down-set generated by the patch closure of that set.) \par \end{enumerate} \end{FACT} \begin{proof} The equivalence of \ref{HeytingEasyIIQCOPheyting} and \ref{HeytingEasyIISpaceHeyting} is in \cite[Theorem 1]{Esakia1974}. The equivalence of \ref{HeytingEasyIISpaceHeyting} and \ref{HeytingEasyIIInvClosed} is in \cite[Theorem 3.1.2]{Esakia2019}, which is a translation of the 1985 Russian original. The equivalence of \ref{HeytingEasyIISpaceHeyting} and \ref{HeytingEasyIISemiOnClosed} is straightforward. \end{proof} \begin{FACT}{:Summary}{pseudoVariantsDiagram} The relationship between Boolean algebras, Heyting algebras, Stone algebras, and pseudocomplemented lattices can be summarized as shown in the diagram below. Recall that ${\color{black}X}=\PrimI({\color{black}L})$, ${\color{black}L}\cong \qcop({\color{black}X})=\{\text{compact open sets}\}$ and $\CK({\color{black}X})=\{\text{constructible sets}\}$. Let $\mathrm{Cl}:\powerset(X)\lra \powerset(X)$ be the closure operator of the powerset of $X$ given by the topology of $X$. Also recall that $L$ is a Boolean algebra if and only if every compact open set is closed. \begin{center} \scalebox{0.75}
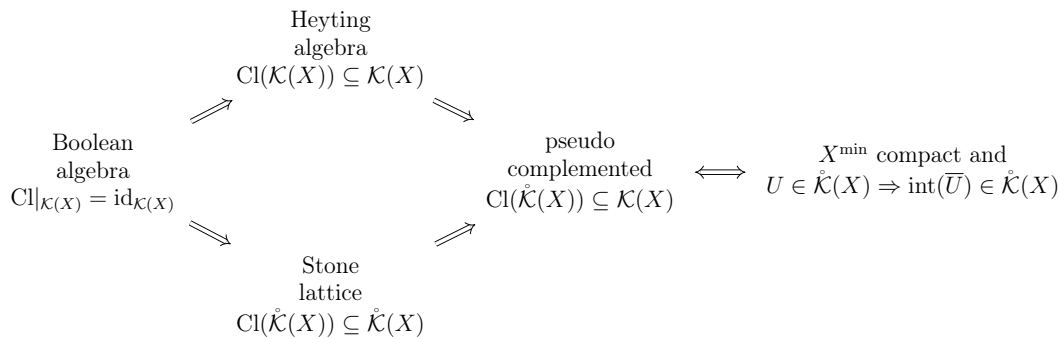
{ \begin{tikzcd}[row sep=2ex,column sep=2ex,ampersand replacement=\&] \& \begin{tabular}{c} {\color{black}\text{Heyting}} \\ {\color{black}\text{algebra}}\\ {\color{black}$\mathrm{Cl}(\CK(X))\subseteq \CK(X)$}\end{tabular} \ar[rd,Rightarrow] \& \&[3ex] \ \\ \begin{tabular}{c} {\color{black}\text{Boolean}} \\ {\color{black}\text{algebra}}\\ {\color{black}$\mathrm{Cl}|_{\CK(X)}=\id_{\CK(X)}$}\end{tabular} \ar[ru,Rightarrow ]\ar[rd,Rightarrow ] \& \& \begin{tabular}{c} {\color{black}\text{pseudo}} \\ {\color{black}\text{complemented}}\\ {\color{black}$\mathrm{Cl}(\qcop(X))\subseteq \CK(X)$}\end{tabular} \ar[r,Leftrightarrow] \& \begin{tabular}{c} {\color{black}$X^{\min}$\text{ compact and}}\\ {\color{black}$U\in\qcop(X)\Ra \interior({\overline{U}})\in\qcop(X)$}\end{tabular}\\ \& \begin{tabular}{c} {\color{black}\text{Stone}} \\ {\color{black}\text{lattice}}\\ {\color{black}$\mathrm{Cl}(\qcop(X))\subseteq \qcop(X)$}\end{tabular}\ar[ru,Rightarrow ] \& \& \\ \end{tikzcd} } \end{center} \end{FACT} \section{Forests and root systems} \labelx{sectionTreeRoot} \noindent To characterize pseudocomplementation in prime spectra of rings of continuous functions, we will need to look at spectral root systems, i.e.~spectral spaces for which every patch closed subset is normal. In this section we collect some information about such spaces and consequences of section \ref{sectionPC} for the pseudocomplemented context.\looseness=-1 \begin{FACT}{:Root systems and forests}{} \begin{enumerate}[{\rm (i)}] \item A spectral space is a \textbf{spectral root system} if the closure of every point is a specialization chain. In older literature this is called \textit{completely normal}, but this conflicts with the standard use of ``completely normal'' in topology, see \DST[(iv)]{herNormalNotComplNormal} and the footnote of \DST{inverseTopOfSinfty}. If we read specialization as a partial order $\leq $, this means that the specialization poset defined by $X$ is a root system. By \DST{rootSysCharact1} a spectral space is a spectral root system iff every patch closed subspace is a normal topological space.\looseness=-1 \item Dually, $X$ is a \notion[]{spectral forest} if $X_\inv$ is a spectral root system. \end{enumerate} \end{FACT} \begin{FACT}{:Stranded spaces}{} \smallskip\noindent A T$_0$-space is \textbf{stranded} if its specialization poset is a sum of chains in the category of posets; in other words, if the specialization relation is a forest and a root system. In particular, stranded spectral spaces are spectral root systems and spectral forests. \end{FACT} \begin{FACT}{:Remark}{} \noindent Trees, forests, and root systems play an important role in modal logic. \begin{enumerate}[{\rm (i)}] \item First completeness results in modal logic with respect to trees go back to \cite{DumLem1959} and \cite{Kripke1963}. Extensions of the intuitionistic propositional calculus IPC complete with respect to forests were investigated by Drugush (see, e.g., \cite{Drugus1982,Drugus1984}). As is argued in \cite{Vardi1997}, it is the tree model property that is responsible for decidability of model-checking in modal logic. For further results in this direction, we refer to \cite{ChaZak1997} and \cite{BldRVe2001}. \item The Heyting algebras whose dual spectra are root systems are known as \textit{G\"odel algebras} (see, e.g., \cite[section 4.2]{Hajek1998}) The corresponding logic was introduced in \cite{Dummet1959} under the name LC, and is now known as the {\it G\"odel-Dummett logic}. This is a prominent extension of IPC which has been studied extensively (see, e.g., \cite{ChaZak1997,Hajek1998}). Its bi-intuitionistic version bi-LC is exactly the logic of stranded posets \cite{BeMaMo2024}. Free algebras in the corresponding varieties of G\"odel algebras have been thoroughly investigated. We reference \cite{BeMaMo2024} and \cite{Carai2026} for a detailed account and relevant references. \par \end{enumerate} \end{FACT} \begin{FACT}{Proposition}{RootForestHeyting} \begin{enumerate}[{\rm (i)}] \item If $X$ is a spectral root system, then the following conditions are equivalent. \begin{enumerate}[(a)] \item $X_\inv$ is an Esakia space. \item $U^{\max}$ is patch closed for all $U\in \qcop(X)$. \item For all $U,V\in \qcop(X)$ the set $U^{\max}\cap V$ is a compact subset of $X$. \end{enumerate} \item If $X$ is a spectral forest, then $X$ is Esakia if and only if $C^{\min}$ is compact for all $C\in \qccl(X)$. \par \end{enumerate} \end{FACT} \begin{proof} (i) Since $X$ is a spectral root system, each $U\in \qcop(X)$ is a normal spectral space. Hence \begin{align*} X_{\inv}\text{ is Esakia}&\iff \text{each }A\in \qccl(X_\inv)\text{ is a PC-space, by \DST{HeytingEasy}},\cr &\iff \text{for all }U\in \qcop(X)\text{ the space }U_\inv\text{ is a PC-space},\cr &\iff \text{for all }U\in \qcop(X)\text{ the set }U^{\max}\text{ is patch closed},\cr &\hspace{20ex} \text{by \ref{QCCLStoneLattice}\ref{QCCLStoneLatticeNormalCoHeyting}$\Leftrightarrow$\ref{QCCLStoneLatticeNormProCon} using that }U\text{ is normal},\cr &\iff \text{for all }U,V\in \qcop(X)\text{ the set }U^{\max}\cap V\text{ is compact,}\cr &\hspace{20ex} \text{by \DST{XmaxProConNormal}}. \end{align*} \noindent (ii) is dual to (i), taking into account that the set of minimal points of a spectral space is patch closed if and only if it is compact. \end{proof} \begin{FACT}{Observation}{HeytingWhenXmaxInCLXmin} Let $X$ be a spectral space. \begin{enumerate}[{\rm (i)}] \item Suppose that every maximal point of $X$ is in the patch closure of the set of minimal points of $X$. Then the following conditions on $X$ are equivalent: \begin{center} (a) $X$ is Boolean, (b) $\qcop(X)$ is a Stone algebra, (c) $X$ is Esakia, \par (d) $X$ is a PC-space, (e) $X^{\min}$ is patch closed. \end{center} \item Suppose that every minimal point of $X$ is in the patch closure of the set of maximal points of $X$.\footnote{Using \DST{minimalPointsCorollary}, this condition is equivalent to saying that $X^{\max}$ is dense in $X$.} Then the following conditions on $X$ are equivalent: \begin{center} (a) $X$ is Boolean, (b) $\qccl(X)$ is a Stone algebra, (c) $X_\inv$ is Esakia, \par (d) $X_\inv$ is a PC-space, (e) $X^{\max}$ is patch closed. \end{center} \end{enumerate} \end{FACT} \begin{proof} Item (ii) is (i) spelled out for $X_\inv$. For item (i), first note that the implications (a)$\Ra $(b),(c) and (b)$\Ra $(d), (c)$\Ra$(d) are obvious, and the implication (d)$\Ra $(e) holds by \ref{pscompBasic}\ref{pscompBasicCharPC}. The implication (e)$\Ra $(a) follows from the assumption that all maximal points of $X$ are in the patch closure of $X^{\min}$, hence there are no specializations in $X$ under assumption (e), and so $X$ is Boolean. \end{proof} \par \noindent The following consequence of our results so far -- displayed for spectral spaces -- will give us in \ref{normalXcoHeytingCT} full information about pseudocomplementation properties for inverse spectra of rings of continuous functions: \belowpdfbookmark{CoHeyting Root Systems are Boolean}{BookmarknormalXcoHeyting} \begin{FACT}{Theorem}{normalXcoHeyting} Let $X$ be a spectral root system such that every minimal point of $X$ is in the patch closure of the set of maximal points of $X$. The following conditions are equivalent. \begin{enumerate}[{\rm (i)}] \item\labelx{normalXcoHeytingXPspace} $X$ is Boolean. \item\labelx{normalXcoHeytingZStone} $\qccl(X)$ is a Stone algebra. \item\labelx{normalXcoHeytingZHeyting} The inverse space of $X$ is Esakia, hence $\qccl(X)$ is a Heyting algebra. \item\labelx{normalXcoHeytingZsHeyting} The inverse space of $X$ is a PC-space, hence $\qccl(X)$ is pseudocomplemented. \item\labelx{normalXcoHeytingMaxProcon} $X^{\max}$ is patch closed. \item\labelx{normalXcoHeytingUmaxCompact} For all $U\in \qcop(X)$ the set $U^{\max}$ is patch closed. \item\labelx{normalXcoHeytingUmaxVCompact} For all $U,V\in \qcop(X)$ the set $U^{\max}\cap V$ is compact. \item\labelx{normalXcoHeytingGenCopen} For every $C\in \qccl(X)$ the set ${\downarrow}C$ is open (equivalently: is patch open) in $X$. \item\labelx{normalXcoHeytingGenConOpen} For every $C\in \qccl(X)$ the set ${\downarrow}C$ is clopen in $X$. \end{enumerate} \par \end{FACT} \begin{proof} Conditions \ref{normalXcoHeytingXPspace}--\ref{normalXcoHeytingMaxProcon} are equivalent by \ref{HeytingWhenXmaxInCLXmin}(ii). Conditions \ref{normalXcoHeytingZHeyting}, \ref{normalXcoHeytingUmaxCompact} and \ref{normalXcoHeytingUmaxVCompact} are equivalent by \ref{RootForestHeyting}. Conditions \ref{normalXcoHeytingZStone}, \ref{normalXcoHeytingZsHeyting} and \ref{normalXcoHeytingGenCopen} are equivalent by \ref{pscompBasic} applied to $X_\inv$. The implication \ref{normalXcoHeytingGenConOpen}$\Ra $\ref{normalXcoHeytingGenCopen} is a weakening. \par \smallskip\noindent \ref{normalXcoHeytingGenCopen}$\Ra$\ref{normalXcoHeytingGenConOpen}. Since $X$ is a spectral root system, it is normal, hence we may apply \ref{NormalSpec}, which tells us that ${\downarrow}C=r^{-1}(C)$ is closed. It is open by \ref{normalXcoHeytingGenCopen}, thus it is clopen. \end{proof} \section{Pseudocomplementation in rings of continuous functions}\label{sectionRings} \noindent In this final section we prove our main results on rings of continuous functions. We start with a reminder on the structure of $\Spec C(T)$. \par \color{black} \begin{FACT}{:Fact}{ContF2} Let $T$ be a completely regular space. Recall from \ref{ContF1} that we write $C(T)$ for the ring of continuous functions $T\lra \R$. \begin{enumerate}[{\rm (i)}] \item\labelx{ContFCinBeta} (See, e.g., \cite[section~5]{Schwar1997}) The embedding $T\into \beta T$ induces an embedding $\iota: C(\beta T)\into C(T)$ of rings (via restriction), its image is the ring $C^*(T)$ of bounded continuous functions $T\lra \R$ and $\Spec(\iota):\Spec C(T)\lra \Spec C(\beta T)$ is a homeomorphism onto an inversely closed (i.e., patch closed and closed under generalizations) subset $S$ of $\Spec C(\beta T)$. This embedding extends the embedding $T\into \beta T$ (when we identify $T$ with its image in $\Spec(C(T))$ as in \ref{ContF1}\ref{ContFeps}) (and similarly for $\beta T$). We get a commutative diagram of embeddings of spaces: \begin{center} \begin{tikzcd}[row sep=8ex,column sep=10ex] \Spec C(T) \ar[r,hook,"\Spec(\iota)"] & \Spec C(\beta T)\\ T \ar[u, hook] \ar[r,hook] & \beta T \ar[u,hook] \end{tikzcd} \end{center} Furthermore, the map that sends a maximal ideal $\Dm$ of $C(T)$ to the Jacobson radical of $\iota^{-1}(\Dm)$ is a homeomorphism between the space of maximal ideals of $C(T)$ and $C(\beta T)$ (Gelfand-Kolmogoroff theorem, available in a much broader context, see \cite[Thm.~10.1]{Tressl2007}); the embedding $\beta T\into \Spec C(\beta T)$ is a homeomorphism onto the space of maximal ideals. \item \labelx{ContFstranded} Finally, the space $(\Spec C(\beta T)) \setminus S$ is stranded (but is in general not spectral) and is equal to the set of proper generalisations of the image of $\Spec(\iota)$. More precisely, the maximal chains in $(\Spec C(\beta T)) \setminus S$ are all of the form \[ \{\Dp\in \Spec C(\beta T)\st \iota^{-1}(\Dm)\subsetneq \Dp\}, \] where $\Dm$ runs through the maximal ideals of $C(T)$. This follows from \cite[Thm. 10.5]{Tressl2007} and also from results in \cite{Schwar1997}. \par Here is a picture of $\Spec C(\beta T)$ showing how $\Spec C(T)$ (identified with the image of $\Spec(\iota)$, shown in the blue area) sits inside $\Spec C(\beta T)$. Closed points (=maximal ideals=maximal points in the spectral space terminology) sit on top, ``specialization goes upwards''; $\Spec C(T)$ and $\Spec C(\beta T)$ have homeomorphic minimal spectra. 
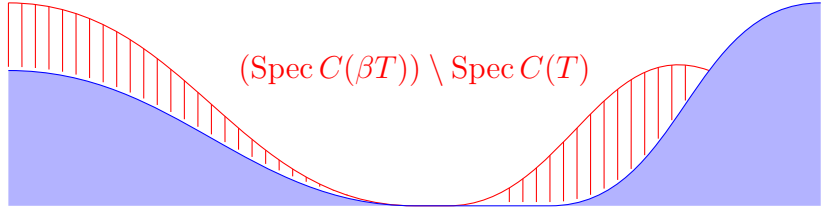
\vspace{-3ex} \begin{center} \begin{tikzpicture}[scale=0.9] \draw[-,red,name path=closedPointsRed] (-6,3) to[out=0, in=180] (0,0) -- (0.5,0) to[out=0, in=160] (4.35,2) ; \draw[-,blue,name path=closedPointsBlue] (-6,2) to[out=0, in=180] (0,0) -- (2,0) to[out=0, in=180] (6,3) ; \fill [opacity=0.3,blue] (-6,2) to[out=0, in=180] (0,0) -- (2,0) to[out=0, in=180] (6,3) -- (6,0) -- (-6,0) -- (-6,2) -- cycle; \foreach \r in {-6,-5.8,...,-1,1.4,1.6,...,4.2}{ \draw [name path=Vline,draw=none] (\r,0) -- (\r,4); \path [name intersections={of=closedPointsBlue and Vline,by=maxPointBlue}]; \path [name intersections={of=closedPointsRed and Vline,by=maxPointRed}]; \draw [-,red,shorten <=0.05cm] (maxPointBlue) -- (maxPointRed); } \node[red] at (0,2) {$(\Spec C(\beta T))\setminus \Spec C(T)$}; \end{tikzpicture} \end{center} \par \end{enumerate} \par \end{FACT} \medskip\noindent For a completely regular space $T$, let $Y$ be one of the spaces \[ \Spec (C(T)), \quad \Spec (C(T))_\inv, \quad \zSpec (C(T)), \quad \mbox{or} \quad \zSpec (C(T))_\inv. \] We will characterize the four conditions on $\qcop(Y)$ in the diagram of \ref{pseudoVariantsDiagram} as well as the condition ``$Y^{\min}$ is compact'' by means of topological properties of $T$. \par We start with $\Spec (C(T))_\inv$ and $\zSpec (C(T))_\inv$. These are characterized by $T$ being a P-space and is deduced from \ref{normalXcoHeyting}: \begin{FACT}{Theorem}{normalXcoHeytingCT} Let $T$ be a completely regular space, $X=\Spec C(T)$ and $Z=\zSpec C(T)$. The following conditions are equivalent. \begin{enumerate}[{\rm (i)}] \item\labelx{normalXcoHeytingCTXPspace} $T$ is a P-space, i.e. $X$ is Boolean and $\qcop(X)$ is a Boolean algebra. \item\labelx{normalXcoHeytingCTZPspace} $Z$ is Boolean, hence $\qccl(Z)$\footnote{Recall from \ref{ContF1}\ref{ContFz} that $\qccl(\zSpec C(T))$ is the order dual of $\Coz(T)$.} is a Boolean algebra. \item\labelx{normalXcoHeytingCTZStone} $\qccl(Z)$ is a Stone algebra. \item\labelx{normalXcoHeytingCTZHeyting} The inverse space of $Z$ is Esakia, hence $\qccl(Z)$ is a Heyting algebra. \item\labelx{normalXcoHeytingCTZsHeyting} The inverse space of $Z$ is a PC-space, hence $\qccl(Z)$ is pseudocomplemented. \item\labelx{normalXcoHeytingCTMaxProcon} $Z^{\max}$ is patch closed. \item\labelx{normalXcoHeytingCTUmaxCompact} For all $U\in \qcop(Z)$ the set $U^{\max}$ is patch closed. \item\labelx{normalXcoHeytingCTUmaxVCompact} For all $U,V\in \qcop(Z)$ the set $U^{\max}\cap V$ is compact. \item\labelx{normalXcoHeytingCTGenCopen} For every $C\in \qccl(Z)$ the set ${\downarrow}C$ is open (equivalently: is patch open) in $Z$. \item\labelx{normalXcoHeytingCTGenConOpen} For every $C\in \qccl(Z)$ the set ${\downarrow}C$ is clopen in $Z$. \end{enumerate} \par \noindent Furthermore these conditions are equivalent to every condition \ref{normalXcoHeytingCTZPspace}--\ref{normalXcoHeytingCTGenConOpen} when they are formulated for $X$ instead of $Z$. \end{FACT} \begin{proof} We have $X^{\min}\subseteq \overline{X^{\max}}^\con=Z$ by \ref{ContF1}\ref{ContFz}, in particular $Z^{\max}=X^{\max}$ and its patch closure contains $Z^{\min}=X^{\min}$. Therefore, by \ref{normalXcoHeyting}, condition \ref{normalXcoHeytingCTZPspace} is equivalent to each of the conditions \ref{normalXcoHeytingCTZStone}--\ref{normalXcoHeytingCTGenConOpen}. Similarly, condition \ref{normalXcoHeytingCTXPspace} is equivalent to each of the conditions \ref{normalXcoHeytingCTZStone}--\ref{normalXcoHeytingCTGenConOpen} formulated for $X$ instead of $Z$. \par Hence it remains to show that \ref{normalXcoHeytingCTXPspace} is equivalent to \ref{normalXcoHeytingCTZPspace}. But this follows from $X^{\min},X^{\max}\subseteq Z$ and the fact that a spectral space is Boolean precisely when all points are minimal and maximal. \end{proof} \par \goodbreak \noindent We now focus on $\Spec (C(T))$ and $\zSpec (C(T))$. \begin{FACT}{void}{RemarksCozeroComplemented}\textbf{Cozero complementation}\quad A completely regular space is called \notion[]{cozero complemented} if for every cozero set $U$ of $T$ there is a cozero set $V$ of $T$ with $U\cap V=\0$ such that $U\cup V$ is dense in $T$. Since $\Coz(T)\cong \qcop(\zSpec C(T))$ and $T$ is patch dense in $\zSpec C(T)$ (see \ref{ContF1}\ref{ContFz}), this is equivalent to saying that for every open and \qc\ subset of $\zSpec C(T)$ there is some open and \qc\ subset $V$ of $\zSpec C(T)$ with $U\cap V=\0$ such that $U\cup V$ is dense in $\zSpec C(T)$. Hence by \ref{cozeroCompInSpec}, $T$ is cozero complemented if and only if $(\zSpec C(T))^{\min}$ is \qc. \par \noindent The terminology and the original proof of the equivalence here is from \cite{HenJer1965} and various further equivalent conditions may be found in \cite[Theorem 1.3]{HenWoo2004}. For cozero complementation in more general contexts we refer to \cite{BhDrMc2011,KLMS2009,MarZen2003,SchTre2010}. \par \smallskip We see from \ref{pscompBasic} that cozero complementation of $T$ is a weakening of $\Spec C(T)$ being a PC-space, which is obtained from dropping condition \ref{pscompBasicCharPC}\ref{pscompBasicCharPCreg} there. \par \end{FACT} \begin{FACT}{:Pseudocomplementation of cozero sets}{pcZ} Let $Z=\zSpec C(T)$. Recall that $\qcop(\zSpec C(T)) \cong \Coz(T)$, and hence all pseudocomplementation conditions considered here are on $\Coz(T)$. Now $\Coz(T)$ is pseudocomplemented provided for each cozero set $U$ of $T$ there is a largest cozero set $V$ of $T$ with $U\cap V=\0$. Using complete regularity of $T$, this is the same as saying that $T\setminus \overline{U}$ itself is a cozero set. Hence, \[ \zSpec(C(T))\text{ is a PC-space}\iff \forall\, U\in \Coz(T): \overline{U}\text{ is a zero set}. \] Such spaces are called \textit{weak Oz-spaces} in \cite{Aull1984}, see also \cite{BDGW2009}. \par For instance, \textit{Oz-spaces} in the sense of \cite{Blair1976} (like perfectly normal spaces, i.e., those spaces for which $\Coz(T)=\CO(T)$) and metric spaces are weak Oz spaces. \end{FACT} \smallskip\noindent In order to characterize when $\Spec C(T)$ is a PC-space, we require the following: \par \belowpdfbookmark{When Spec C(T) is a PC-space}{BookmarkCofXConvexSemiHeyting} \begin{FACT}{Proposition}{CofTConvexSemiHeyting} Let $T$ be a completely regular space and let $X\subseteq \Spec C(T)$ be patch closed and convex for specialization, i.e.~$\Dp\subseteq \Dq\subseteq \Dr$ and $\Dp,\Dr\in X$ imply $\Dq\in X$. The following are equivalent: \begin{enumerate}[{\rm (i)}] \item\labelx{CofTConvexSemiHeytingPC} $X$ is a PC-space. \item \labelx{CofTConvexSemiHeytingStoneI} $\qcop(X)$ is a Stone algebra. \item\labelx{CofTConvexSemiHeytingStoneII} $X^{\min}$ is patch closed and $X$ is stranded. \end{enumerate} \end{FACT} \begin{proof} \ref{CofTConvexSemiHeytingStoneII}$\Leftrightarrow$\ref{CofTConvexSemiHeytingStoneI} follows from \ref{QCOPStoneLattice} because the spectral root system $X$ is stranded iff $X_\inv$ is normal. The implication \ref{CofTConvexSemiHeytingStoneI}$\Ra $\ref{CofTConvexSemiHeytingPC} holds by definition. \par \smallskip\noindent \ref{CofTConvexSemiHeytingPC}$\Ra $\ref{CofTConvexSemiHeytingStoneII} Assume that $X$ is a PC-space. By \ref{pscompBasic} we know that $X^{\min}$ is patch closed. Suppose $X$ is not stranded. Since $X$ is a spectral root system, there must be points $\Dp\neq \Dq$ in $X^{\min}$ and a common specialization of $\Dp$ and $\Dq$ in $X$. Because $X$ is convex, we know $1\notin \Dp+\Dq$. By \cite[Proposition 3.9]{Tressl2006} (and \cite[14B]{GilJer1960}), the sum $\Dp+\Dq$ of the ideals $\Dp$ and $\Dq$ is a z-prime ideal. As $X$ is convex, we get $\Dp+\Dq\in X$. Since $X^{\min}$ is Hausdorff, there is some $U\in\qcop(X)$ with $\Dq\in U$ and $\Dp\notin U$. Since $\Dp$ is minimal in $X$, we get $\Dp\notin \Spez(U)\cap X=\overline{U}\cap X$. Let $C=\overline{\{\Dp\}}\cap \overline{U}\cap \Gen(\Dp+\Dq)$. Since $\Spec C(T)$ is a spectral root system, $C$ is a chain, which is not empty because $\Dp+\Dq\in C$. Because $X$ is convex, and $\Dp,\Dp+\Dq\in X$, we know $C\subseteq X$. \par Now, the nonempty patch closed chain $C$ has a smallest element, which we denote by $\Dr$. As $\Dr\in \overline{U}$ there is some $\Dq_0\in U\cap X^{\min}$ with $\Dq_0\subseteq \Dr$. As $\Dp\notin U$, the prime ideals $\Dp$ and $\Dq_0$ are incomparable. By \cite[Proposition 3.9]{Tressl2006} again, $\Dp+\Dq_0$ is a z-radical prime ideal of $C(T)$. Therefore, $\Dp+\Dq_0\in \Gen(\Dr)\cap C$ and the minimality of $\Dr$ in $C$ implies that $\Dr=\Dp+\Dq_0$ is a z-radical prime ideal with $\Dr\in X$. Let \[ S=\bigl(\overline{\{\Dp\}}\cap\Gen(\Dr)\bigr)\setminus \{\Dr\}. \] Here is the depiction of the situation, where the lines represent inclusion, the blue points are in $U$ and the green points are in $C$. 
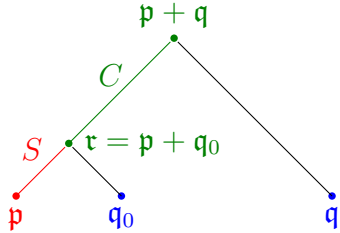
\begin{center} \begin{tikzpicture}[scale=0.7] \coordinate (p) at (0,0); \coordinate (q0) at (2,0); \coordinate (q) at (6,0); \coordinate (r) at (1,1); \coordinate (pq) at (3,3); \fill[red] (p) circle (2pt) node[below] {$\Dp$}; \fill[blue] (q0) circle (2pt) node[below] {$\Dq_0$}; \fill[blue] (q) circle (2pt) node[below] {$\Dq$}; \fill[Green] (r) circle (2pt) node[right] {\,$\Dr=\Dp+\Dq_0$}; \fill[Green] (pq) circle (2pt) node[above] {$\Dp+\Dq$}; \draw[red,shorten >=2pt] (p) -- (r); \draw[black,shorten >=2pt] (q0) -- (r); \draw[black,shorten >=2pt] (q) -- (pq); \draw[Green,shorten >=2pt] (r) -- (pq); \node[red] at (0.3,0.9) {$S$}; \node[Green] at (1.8,2.3) {$C$}; \end{tikzpicture} \end{center} \par \noindent Since $\Dp\neq \Dr$, we get $\Dp\in S\subseteq \overline{\{\Dp\}}$, hence $S$ is a nonempty chain. Since $\Dp,\Dr\in X$ and $X$ is convex, we also see that $S\subseteq X$. The minimality of $\Dr$ in $C$ implies $S\cap \overline{U}=\0$, hence $S\subseteq X\setminus \overline{U}$. Since $X$ is a PC-space, the set $\overline{U}\cap X$ is constructible in $X$, so $X\setminus \overline{U}$ is patch closed. Because $\Dr\in X\cap \overline{U}$, we obtain $\Dr\notin \overline{\,S\,}^\con$. \par We now show that $\Dr\in \overline{\,S\,}^\con$, which gives the desired contradiction. Since $S$ is a chain, it suffices to show that $S$ has no largest element with respect to inclusion because then the chain $S$ has the supremum $\Dr$ and this has to be in $\overline{\,S\,}^\con$. \par In order to see that $S$ has no largest element, let $I\in \Spec(C(T))$ and let $f\in C(T)$ with $I\subseteq \Dr$ and $f\in \Dr\setminus I$. Then $I\subsetneq J=\sqrt{I+f\mal C(T)}\in \Spec (C(T))$, but $J\neq \Dr$ by \cite[Lemma 14.1]{Tressl2007} since $\Dr$ is a z-ideal. \end{proof} \par \goodbreak \begin{FACT}{Theorem}{CofTSemiHeyting} Let $T$ be a completely regular space. The following are equivalent: \begin{enumerate}[{\rm (i)}] \item\labelx{CofTSemiHeytingSH} $\Spec(C(T))$ is a PC-space. \item\labelx{CofTSemiHeytingStone} $\qcop(\Spec C(T))$ is a Stone algebra. \item\labelx{CofTSemiHeytingFspaceQC} $\Spec(C(T))^{\min}$ is patch closed and $T$ is an $F$-space, i.e.~$\Spec(C(T))$ is stranded.\footnote{See \cite[Theorem in 14.25]{GilJer1960} for more information on F-spaces.} \item\labelx{CofTSemiHeytingBasDisc} $T$ is basically disconnected. \item\labelx{CofTSemiHeytingdedekind} The poset $C(T)$ is Dedekind $\sigma$-complete. \item\labelx{CofTSemiHeytingboundedIntervals} For every $f>0$ in $C(T)$, the interval $[0,f]$ is pseudocomplemented. \end{enumerate} \end{FACT} \noindent \texttt{Remark.} The equivalence of the last four conditions is not new and is recorded here to showcase the topological impact of the first two conditions. \begin{proof} \ref{CofTSemiHeytingSH}$\Leftrightarrow $\ref{CofTSemiHeytingStone}$\Leftrightarrow $\ref{CofTSemiHeytingFspaceQC} holds by \ref{CofTConvexSemiHeyting} applied to $X=\Spec C(T)$. The equivalence \ref{CofTSemiHeytingFspaceQC}$\Leftrightarrow$\ref{CofTSemiHeytingBasDisc} essentially is the equivalence of the conditions \ref{QCOPStoneLatticeNormProCon} and \ref{QCOPStoneLatticeStone} in \ref{QCOPStoneLattice} applied to the spectral root system $\zSpec (C(T))$ because the definition of basically disconnected says that $\Coz(T)$ is a Stone algebra and $\zSpec (C(T))$ and $\Spec (C(T))$ have the same minimal spectrum and one of them is stranded iff the other one is. However, our assertion \ref{CofTSemiHeytingFspaceQC}$\Leftrightarrow$\ref{CofTSemiHeytingBasDisc} here is not new: Recall from \cite[Theorem 5.3(e)]{HenJer1965} that for any $F$-space $T$, $\Spec(C(T))^{\min}$ is patch closed if and only if $T$ is basically disconnected.\footnote{Observe that any basically disconnected space is cozero complemented: If $U$ is a cozero set, then $\overline{U}$ is open, hence is clopen and so is the zero set of its characteristic function. In comparison: $\Spec(C(T))$ is Boolean if and only if all cozero sets are closed, cf. \cite[14.29]{GilJer1960}.} This shows \ref{CofTSemiHeytingFspaceQC}$\Ra$\ref{CofTSemiHeytingBasDisc}. Conversely, by \cite[14N, p.~215]{GilJer1960}, every basically disconnected space is an $F$-space, hence \ref{CofTSemiHeytingBasDisc} also implies \ref{CofTSemiHeytingFspaceQC}. \par \smallskip\noindent The equivalence of \ref{CofTSemiHeytingBasDisc} with each of \ref{CofTSemiHeytingdedekind} and \ref{CofTSemiHeytingboundedIntervals} is not new and belongs to the theory of Abelian $\ell$-groups: In \cite[pp.~283--287]{LuxZaa1971} it is shown that a completely regular space $T$ is basically disconnected (named `principal projection property' in that source) if and only if $C(T)$, seen as a poset, is Dedekind $\sigma$-complete, if and only if $C(T)$, viewed as a lattice ordered group, is \textbf{projectable}, i.e.~for every $f>0$ in $C(T)$ the interval $[0,f]$ is pseudocomplemented. For a reference and some explanation of the terminology, see \cite[Proposition 2.1]{Wynne2007}. \end{proof} \begin{FACT}{Corollary}{SpecCbetaXSemiH} $\Spec C(T)$ is a PC-space if and only if $\Spec C(\beta T)$ is a PC-space. \end{FACT} \begin{proof} The spaces $\Spec C(T)$ and $\Spec C(\beta T)$ have homeomorphic minimal spectra, see \ref{ContF2}\ref{ContFstranded}. By \cite[Theorem~in~14.25]{GilJer1960}, $T$ is an $F$-space if and only if $\beta T$ is an $F$-space. Hence, the desired equivalence follows from \ref{CofTSemiHeyting}. \par Alternatively one can use \cite[6M, p.~96]{GilJer1960}, which says that $T$ is basically disconnected if and only if $\beta T$ is basically disconnected and then deploy \ref{CofTSemiHeyting}. \end{proof} \begin{FACT}{Corollary}{CofXnotHeytingConvSeqNormal} Let $T$ be a completely regular space. If $\Spec C(T)$ is a PC-space, then no sequence $(x_n)_{n\in\N}$ of distinct points of $T$ has a limit in $T$. In particular, each metrizable subspace of $T$ is discrete. \end{FACT} \begin{proof} By \GilJer{224}{14N.1}, no point of an F-space is the limit of a sequence of distinct points. Now apply \ref{CofTSemiHeyting}. \end{proof} \par \noindent As a consequence of \ref{CofXnotHeytingConvSeqNormal} we obtain: \begin{FACT}{Corollary}{semiHeytingMetric} If $T$ is a metric space, then $\Spec C(T)$ is a PC-space if and only if it is an Esakia space if and only if $T$ is discrete. \end{FACT} \begin{FACT}{Corollary}{PspaceViaVfNotOpeninMin} Let $T$ be a completely regular space. Then $T$ is a $P$-space if and only if $T$ is basically disconnected and ${\downarrow}V(f)\cap (\Spec C(T))^{\min}$ is open in $(\Spec C(T))^{\min}$ for every $f\in C(T)$. \end{FACT} \begin{proof} If $T$ is a $P$-space, then $T$ is basically disconnected (see \ref{ContF1}) and ${\downarrow}V(f)=V(f)$ is clopen. For the converse, we work in $X=\Spec C(T)$ and use \ref{normalXcoHeyting} to show that ${\downarrow}V(f)$ is constructible for all $f\in C(T)$: By \ref{CofTSemiHeyting} we know that $T$ is an $F$-space and $X^{\min}$ is patch closed. It follows that the restriction $r_0:X^{\min}\lra X^{\max}$ of the natural retraction $r:X\lra X^{\max}$ is a bijective continuous map between compact spaces. Hence $r_0$ is a homeomorphism. \par Since ${\downarrow}V(f)$ is closed, the assumption implies that ${\downarrow}V(f)\cap X^{\min}$ is clopen in $X^{\min}$. Consequently, the set $V(f)\cap X^{\max}=r_0({\downarrow}V(f)\cap X^{\min})$ is clopen in $X^{\max}$. Since $r$ is a spectral map $X\lra X^{\max}$ (see \DST{XmaxBooleanNormal} and use that $X^{\max}$ is Boolean), this shows that ${\downarrow}V(f)=r^{-1}(V(f)\cap X^{\max})$ is constructible as well. \end{proof} \begin{FACT}{:Pseudocomplementation conditions for some special spaces}{ExamplesOverview} \noindent Let $S$ be the space exhibited in \cite[Example 5.8]{HenJer1965}, where it is shown that $S$ is cozero complemented but $\Coz(S)$ is not pseudocomplemented. By \cite[Example 5.8]{HenJer1965}, the space $\beta \N\setminus \N$ does not have compact minimal spectrum, in fact no point of $(\Spec C(\beta \N\setminus \N))^{\min}$ has a compact neighborhood. Since $\zSpec C(\beta \N\setminus \N)$ is naturally homeomorphic to a closed and constructible subset of $\zSpec C(\beta \N)$\footnote{By the Tietze extension theorem, the restriction map $C(\beta \N)\lra C(\beta \N\setminus \N)$ is surjective and one verifies without difficulty that its kernel is the $z$-ideal $I$ generated by the unique extension of the function $\frac{1}{n}$ from $\N$ to $\beta\N$; hence $\zSpec C(\beta\N\setminus \N)\cong V(I)\cap \zSpec C(\beta\N)$.}, \ref{HeytingEasyII} implies that the latter is not an Esakia space. \par Using this information we present a table with an overview of the pseudocomplementation conditions we have considered. As above, $X=\Spec(C(T))$ for a completely regular space $T$ and $Z=\zSpec(C(T))$. The conditions in rows with an asterisk in the second column are equivalent for all choices of $T$ (not only those shown in the last four columns). Similarly for the conditions in rows with a dagger. The numbers behind an entry give a reference that implies the answer. Note that if $X$ is PC, then also $Z$ is PC as follows from \DST{semiHeytingSubspaces}, which implies that $Z$ is even a PC-subspace of $X$. \par \medskip \begin{center} \providecommand\tableH[1]{$#1$} \providecommand\tableX[1]{#1} \providecommand\tableZ[1]{{\color{black}#1}} \begin{tabular}{|c|c|c|c|c|c|} \hline {\tableH{}} & {\tableH{\Leftrightarrow}} & {\tableH{\R^n}} & {\tableH{\beta \N}} & {\tableH{\beta \N\setminus \N}} & {\tableH{S\text{ from }\ref{ExamplesOverview}}} \\ \hline \tableX{$X$ Stone} & \tableX{$*$} & \tableX{N, \ref{CofTSemiHeyting}} & \tableX{Y, \ref{CofTSemiHeyting}} & \tableX{N, \ref{CofTSemiHeyting}} & \tableX{N, \ref{CofTSemiHeyting}} \\ \hline \begin{tabular}{c} \tableZ{$Z$ Stone, i.e.} \\ \tableZ{$T$ basically} \\ \tableZ{disconnected} \end{tabular} & \tableZ{$*$}    & \tableZ{N, \ref{CofTSemiHeyting}} & \tableZ{Y, \ref{CofTSemiHeyting}} & \tableZ{N, \ref{CofTSemiHeyting}} & \tableZ{N, \ref{CofTSemiHeyting}}    \\ \hline \tableX{$X$ Esakia}    & \tableX{}     & \tableX{N, \ref{CofTSemiHeyting}} & \tableX{\textbf{?}} & \tableX{N, \ref{ExamplesOverview}} & \tableX{N, \ref{ExamplesOverview}}     \\ \hline \tableZ{$Z$ Esakia}    & \tableZ{}     & \tableZ{Y, \ref{pcZ}} & \tableZ{N, \ref{ExamplesOverview}} & \tableZ{N, \ref{ExamplesOverview}} & \tableZ{N, \ref{ExamplesOverview}}     \\ \hline \tableX{$X$ PC}       & \tableX{$*$}    & \tableX{N, \ref{CofXnotHeytingConvSeqNormal}} & \tableX{Y, \ref{SpecCbetaXSemiH}} & \tableX{N, \ref{CofTSemiHeyting}} & \tableX{N, \ref{CofTSemiHeyting}}    \\ \hline \tableZ{$Z$ PC}       & \tableZ{}     & \tableZ{Y, \ref{pcZ}} & \tableZ{Y, \ref{SpecCbetaXSemiH}} & \tableZ{N, \ref{CofTSemiHeyting}} & \tableZ{N, \ref{CofTSemiHeyting}}     \\ \hline \tableX{$X^{\min}$ compact} & \tableX{$\dagger$} & \tableX{Y, \ref{RemarksCozeroComplemented}} & \tableX{Y, \ref{ContF2}} & \tableX{N, \ref{ExamplesOverview}} & \tableX{Y, \ref{ExamplesOverview}} \\ \hline \tableZ{$Z^{\min}$ compact} & \tableZ{$\dagger$} & \tableZ{Y, \ref{RemarksCozeroComplemented}} & \tableZ{Y, \ref{ContF2}} & \tableZ{N, \ref{ExamplesOverview}} & \tableZ{Y, \ref{ExamplesOverview}} \\ \hline \end{tabular} \end{center} \par \medskip\noindent Also recall from \ref{normalXcoHeytingCT} that for the inverse spaces of $X$ and $Z$ all conditions are equivalent to $T$ being a P-space. \end{FACT} \smallskip\noindent \textbf{Open problem.} We do not have a characterization of those completely regular spaces $T$ for which $\Spec (C(T))$ is an Esakia space. In \ref{semiHeytingMetric} we have seen that metric spaces have this property only if they are discrete, but we have no example of an infinite compact Hausdorff space $T$ such that $\Spec (C(T))$ is an Esakia space. In particular, we do not know whether $\Spec (C(\beta \N))$ is Esakia. \par We also do not have a general characterization of when $\zSpec (C(T))$ is Esakia (equivalently: $\Coz(T)$ is a Heyting algebra). This does happen frequently, e.g. when all open sets of $T$ are cozero sets, like in perfectly normal spaces, but we do not have a full characterization.  \par \end{document}